\documentclass[11pt]{article}
\usepackage{epsf}
\usepackage{amsbsy,amsmath}
\usepackage{mathtools}
\usepackage{mathrsfs}
\usepackage{amsfonts}
\usepackage{amssymb}
\usepackage{enumitem}
\usepackage{eucal}
\usepackage{enumitem}
\usepackage{graphics,mathrsfs}
\usepackage{amsthm}
\usepackage{secdot}
\usepackage{mathtools}
\usepackage{hyperref}
\usepackage[displaymath, mathlines]{lineno}
\newtheorem{theorem}{Theorem}[section]
\newtheorem{lemma}[theorem]{Lemma}
\newtheorem{proposition}[theorem]{Proposition}

\theoremstyle{definition}
\newtheorem{definition}[theorem]{Definition}

\newtheorem{Claim}[theorem]{Claim}
\theoremstyle{remark}
\newtheorem{remark}[theorem]{Remark}
\numberwithin{equation}{section}

\numberwithin{equation}{section}
\newsavebox{\savepar}

\begin{document}	
	\title{\sc On semilinear equations with free boundary conditions on stratified Lie groups}
	\author{\sc Debajyoti Choudhuri$^{a}$ and
		\sc Du\v{s}an D.  Repov\v{s}$^{b,c,d}$\footnote{Corresponding
			author: dusan.repovs@guest.arnes.si}\\
		\small{$^{a}$Department of Mathematics, National Institute of Technology Rourkela, India}\\
		\small{$^b$Faculty of Education, University of Ljubljana, Ljubljana, Slovenia }\\
		\small{$^c$Faculty of Mathematics and Physics, University of Ljubljana, Ljubljana, Slovenia }\\
		\small{$^{d}$ Institute of Mathematics, Physics and Mechanics, Ljubljana, Slovenia.}
	}
	\date{}
	\maketitle
	\begin{abstract}
		\noindent In this paper we establish existence of a solution to a semilinear equation with  free boundary conditions on stratified Lie groups. In the process, a monotonicity condition is proved, which is quintessential in establishing the regularity of the solution.
		\begin{flushleft}
			{\sl Keywords and Phrases}:~ Dirichlet free boundary value problem, Stratified Lie group, sub-Laplacian.\\
			{\sl Math. Subj. Classif. (2020)}:~35J35, 35J60.
		\end{flushleft}
	\end{abstract}
	\section{Introduction}\label{s1}
	In this paper we shall prove the existence of solutions to the following {\it free boundary value} problem:
	\begin{equation}\label{probmain}
		\left\{\begin{aligned}
			-\mathcal{L} u &=\lambda(u-1)_+^2f,~\text{in}~\Omega\setminus H(u),\\
			|\nabla_{\mathbb{G}} u^+|^{2}-|\nabla_{\mathbb{G}} u^-|^2&=2,~~~~~~~~~~~\text{in}~H(u),\\
			u &= 0, ~~~~~~~~~~~\text{on}~\partial\Omega.
		\end{aligned}\right.
	\end{equation}
	Here, $\lambda>0$, $(u-1)_+=\max\{u-1,0\},$ and 
	$H(u)=\partial\{u>1\}.$ Also,  $\nabla_{\mathbb{G}} u^{\pm}$ are the limits of $\nabla_{\mathbb{G}} u$ for the sets $\{u>1\}$ and $\{u\leq 1\}^{\circ},$ respectively. Next, $f\in L^{\infty}(\Omega)$ is a positive bounded function. The domain $\Omega\subset\mathbb{G}$ is bounded, where $\mathbb{G}$ is a stratified Lie group. Finally, $\mathcal{L}$ is the sub-Laplacian which will be defined in Section~\ref{s2}.\\
	The study of elliptic free boundary value problems (FBVPs)  has recently gained momentum, owing to its rich mathematical content besides its physical applications. A naturally occurring free boundary condition can be found in the
	classical problem in fluid dynamics to model a 2-dimensional ideal fluid in terms of its stream function (see {\sc Dipierro et al.} \cite{dipi1}). Interested readers can also check {\sc Batchelor} \cite{Batchelor1,Batchelor2} for the Prandtl-Batchelor free boundary. \\
	From a mathematical point of view, the problem
		\begin{equation}\label{alt_prob}
	\left\{\begin{aligned}
	-\Delta u&=0,~\text{in}~\Omega\setminus G(u),\\
	|\nabla u^+|^2-|\nabla u^-|^2&=2,~~~~~~~~~~~~~~~~~~~~~~~~~~~~~~~~\text{on}~G(u),\\
	u&=0,~~~~~~~~~~~~~~~~~~~~~~~~~~~~~~~~\text{on}~\partial\Omega
	\end{aligned}\right.
	\end{equation}
	has been studied by {\sc Alt-Caffarelli} \cite{alt_caffa}, {\sc Alt et al.} \cite{2}, {\sc Caffarelli et al.} \cite{Caffa_jeri_kenig,12}, and  {\sc Weiss} \cite{43,44}. Later on, {\sc Jerison-Perera} in \cite{J_P_1,30} considered the problem 
	\begin{equation}\label{perera_prob0}
	\begin{aligned}
	-\Delta u&=(u-1)_+^{p-1},~\text{in}~\Omega\setminus G(u),
	\end{aligned}
	\end{equation}
	in particular with the same bounday conditions as in \eqref{alt_prob}, with $G(u)=\partial\{u>1\}$, thus pioneering the study of the existence of a mountain pass point  at which the associated energy functional has a higher value compared to the global minimum (see  \cite[Definition 1]{hofer}). Such a critical point was referred by them as a {\it higher critical point}. A slightly more general problem was considered by {\sc Perera} in \cite{Perera_NoDEA}, as follows
	\begin{equation}\label{perera_prob}
		\left\{\begin{aligned}
			-\Delta u&=\alpha\chi_{\{u>1\}}(x)f(x,(u-1)_+),~\text{in}~\Omega\setminus G(u),\\
			|\nabla u^+|^2-|\nabla u^-|^2&=2,~~~~~~~~~~~~~~~~~~~~~~~~~~~~~~~~\text{on}~G(u),\\
			u&=0,~~~~~~~~~~~~~~~~~~~~~~~~~~~~~~~~\text{on}~\partial\Omega. 
		\end{aligned}\right.
	\end{equation}
	This problem was also studied by {\sc Elcrat-Miller} \cite{elc1} and {\sc Jerison-Perera} \cite{J_P_1} for the case $N=2$. The main result of \cite{Perera_NoDEA} is the establishment of a higher critical point. Some of the important works in the Euclidean setting have been documented in {\sc Dipierro et al.} \cite{dipi1}
	and {\sc Perera} \cite{Perera_NoDEA}, and the references therein.\\
Motivated by the above mentioned works, albeit in the Euclidean setting, we consider \eqref{probmain} in the non-Euclidean setup. One key work in this direction is  {\sc Ferrari-Valdinoci}~\cite{valdinoci1}, in which a free boundary value problem was studied on the Heisenberg group, and the authors established some density estimates for local minima. The problem which we shall study in this paper is classical, however its consideration over a stratified Lie group is new since the Heisenberg group is also a particular kind of a stratified Lie group.\\
We now state the main result of this paper pertaining to the existence of  solutions to problem  \eqref{probmain}:
\begin{theorem}\label{main_res2}
There exists $\lambda_*>0$ such that for any $0<\lambda<\lambda_*$, there exists a positive solution $u$ to problem \eqref{probmain} with the following properties:
\begin{enumerate}[label=(\roman*)]
	\item $u$ is a critical point of $I$;
	\item $u$ satisfies the free boundary condition in the sense of viscosity.
%	\item there exists a neighborhood of each regular point in which the free boundary\\ $\partial\{u>1\}$ is a $C^{1,\beta}$ - surface and $u$ satisfies the free boundary condition in the classical sense.
\end{enumerate}
\end{theorem}
\begin{remark}\label{nontrivial}Notice that by a nontrivial solution to \eqref{probmain} we mean $u>0$ on $\Omega$ and $u>1$ on a nonempty open subset of $\Omega$ on which  $-\mathcal{L}u=\lambda (u-1)^2f$ holds.\end{remark}
	\noindent
	The organization of the paper is as follows. 
	In Section~\ref{s2} we recall the preliminaries of the stratified Lie group and the space description. In addition, we prepare the necessary tools required to attack problem  \eqref{probmain}. 
	In Section~\ref{s3} we prove a monotonicity lemma (Lemma~\ref{conv_res1}). 
	In Section~\ref{s4} we prove a convergence lemma (Lemma~\ref{convergence1}).
	In Section~\ref{s5} we prove the main result of this paper (Theorem~\ref{main_res2}).
	Finally, in Section~\ref{s6} we prove an auxilliary lemma on the Radon measure (Lemma~\ref{positive_Rad_meas}).
	\section{Preliminaries 
		%Definitions and space description
	}\label{s2}
	This section includes the necessary tools to study  problem \eqref{probmain}. For all other background information we refer to the comprehensive handbook \cite{PRR}. We begin by the definition of a homogeneous Lie group.
	\begin{definition}\label{HLG}
		A Lie group $\mathbb{G}$, on $\mathbb{R}^N$ is said to be homogeneous, if for any $\mu>0$ there exists an automorphism $T_{\mu}:\mathbb{G}\to\mathbb{G}$ defined by $$T_{\mu}(x)=(\mu^{r_1}x_1,\mu^{r_2}x_2,\cdots,\mu^{r_N}x_N),
		\quad
		r_i>0,
		\quad
		i=1,2,\cdots,N.$$ The map $T_{\mu}$ is called a dilation on $\mathbb{G}$. Here, $x=(x_1,x_2,\cdots,x_N)$.
	\end{definition}
	\noindent It is worth noting that  $N$ represents the topological dimension of $\mathbb{G}$, whereas $D=r_1+r_2+\cdots+r_N$ represents the homogeneous dimension of the homogeneous Lie group $\mathbb{G}$. The symbol $dx$ will serve as our notation for
	the  Haar measure, which is the standard Lebesgue measure on $\mathbb{R}^N$. The following is the definition of a stratified Lie group.
	\begin{definition}\label{SLG_defn}
		A homogeneous Lie group $\mathbb{G}=(\mathbb{R}^N,*)$ is called a stratified Lie group (or a homogeneous Carnot group) if the following two conditions are satisfied:
		\begin{enumerate}[label=(\roman*)]
			\item The decomposition $\mathbb{R}^N=\mathbb{R}^{N_1}\times\mathbb{R}^{N_2}\times\cdots\times\mathbb{R}^{N_k}$ holds for some natural numbers $N_1,N_2,\cdots,N_k$ such that $N_1+N_2+\cdots+N_k=N$. Furthermore, for each $\mu>0$ there exists a dilation of the form $T_{\mu}(x)=(\mu x^{(1)},\mu^2 x^{(2)},\cdots,\mu^k x^{(k)})$ which is an automorphism of the group $\mathbb{G}$. Here, $x^{(i)}\in\mathbb{R}^{N_i}$ for each $i=1,2,\cdots,k$.
			\item Let $N_1$ be the same as in the above decomposition of $\mathbb{R}^N$
			and let $X_1,X_2,\cdots,X_{N_1}$ be the left invariant vector fields on $\mathbb{G}$ such that $X_i(0)=\frac{\partial}{\partial x_i}|_{0}$ for $i=1,2,\cdots,N_1$. Then the H\"{o}rmander condition  rank(Lie$\{X_1,X_2,\cdots,X_{N_1}\})=N$ holds for every $x\in\mathbb{R}^N$. Roughly speaking, the Lie algebra corresponding to the Lie group $\mathbb{G}$ is spanned by the iterated commutators of $X_1,X_2,\cdots,X_{N_1}$.
		\end{enumerate}
	\end{definition}
	\noindent Here, $k$ is called the {\it step} of the homogeneous Carnot group. In the case of a stratified Lie group, the homogeneous dimension becomes $D=\sum_{i=1}^{k}iN_i$. 
	%Moreover, the left invariant vector fields $X_i$ obey the divergence theorem which can be expressed as
	%$$X_i=\frac{\partial}{\partial x_i^{(1)}}+\sum_{j=2}^{k}\sum_{l=1}^{N_l}a_{i,l}^{(j)}(x^1,x^2,\cdots,x^{j-1})\frac{\partial}{\partial x_l^{(j)}}.$$
	Throughout the paper, we set $N=N_1$ in Definition \ref{SLG_defn}. We call a curve $\gamma:[0,1]\to\mathbb{R}$   admissible if there exists $c_i:[0,1]\to\mathbb{R}$, for $i=1,2,\cdots,N$ such that 
	$$\gamma'(t)=\sum_{i=1}^{N}c_i(t)X_i(\gamma(t)),~\sum_{i=1}^{N}c_i(t)^2\leq 1.$$
	Here, $\gamma'$ is the derivative with respect to $t$. The functions $c_i$ may not be unique since the vector fields $X_i$ may not be linearly independent. For any $x,y\in\mathbb{G},$ the {\it Carnot-Carath\'{e}odory} distance is defined by 
	$$d_{cc}(x,y)=\inf\{l>0:\text{there exists an admissible}~\gamma:[0,l]\to\mathbb{G}~\text{such that}~\gamma(0)=x, \gamma(l)=y\}.$$
	If no such curve exists, $d_{cc}(x,y)$ is set to $0$.
	Although $d_{cc}$ is not a metric in general, the 
	H\"{o}rmander
	condition over the vector fields $X_{1}, X_{2}, \cdots, X_{N_{1}}$ ensures that it is.
	The space $(\mathbb{G}, d_{cc})$
	is then referred to as the
	{\it Carnot-Carath\'{e}odory}
	space. The definition of the homogeneous quasi-norm on the homogeneous Carnot group $\mathbb{G}$ is another important entity that will be used in the course of this work. See {\sc Ghosh et al.} \cite[Definition 2.3]{ghosh1} for a definition of a homogeneous quasi-norm.\\
	\noindent
	Furthermore, the sub-Laplacian, the horizontal gradient and the horizontal divergence on $\mathbb{G}$ is defined as
	$$\mathcal{L}:=X_1^2+X_2^2+\cdots+X_{N_1}^2,
	\
	\nabla_{\mathbb{G}}:=(X_1,X_2,\cdots,X_{N_1}),
	\
	\text{div}_{\mathbb{G}}v:=\nabla_{\mathbb{G}}\cdot v.$$
	respectively. The sub-Laplacian on the stratified Lie group $\mathbb{G}$ is defined as $\Delta_{\mathbb{G}}u:=div_{\mathbb{G}}(\nabla_{\mathbb{G}}u).$ \\
	Now, let $S$ be a Haar measurable subset of $\mathbb{G}$. Then $\mathcal{H}(T_{\mu}(S))=\mu^{D}\mathcal{M}(S),$ where $\mathcal{H}(S)$ is the Haar measure of $\Omega$. A quasi-ball of radius $r$ and centered at $x\in\mathbb{G}$ is defined by $B(x,r)=\{y\in\mathbb{G}:|y^{-1}*x|<r\}$ with respect to the quasi-norm $|\cdot|$. \\
	We define the Sobolev space,  which is very essential in order to venture into this problem. For $1<p<\infty$, the Sobolev space $W^{1,p}(\Omega)$ on a stratified Lie group is defined as
	\begin{align}\label{sob_spc}
		W^{1,p}(\Omega)&:=\{u\in L^p(\Omega):|\nabla_{\mathbb{G}}u|\in L^p(\Omega)\}.
	\end{align}
	A norm on this space is given by
	$\|u\|_{1,p}:=\|u\|_p+\|u\|.$
	Here, $$\|u\|_p=\left(\int_{\Omega}|u(x)|^pdx\right)^{1/p},~\|u\|:=\left(\int_{\Omega}|\nabla_{\mathbb{G}}u(x)|^pdx\right)^{1/p}.$$
	We 
	define the space $W_0^{1,p}(\Omega)$ as follows:
	$$W_0^{1,p}(\Omega)=\{u\in W^{1,p}(\Omega):u=0~\text{on}~\partial\Omega\},$$ where $u=0~\text{on}~\partial\Omega$ is in the usual trace sense. We note that  $W_0^{1,p}(\Omega)$ is a real separable and uniformly convex Banach space (see  \cite{foll, vodo1,vodo2,xu}) . The following embedding result follows from  \cite[$(2.8)$ ]{7},  \cite{foll}, 
	and 
	\cite[Theorem
	$8.1$]{14}. We also suggest the reader to check \cite[Theorem $2.3$]{2}.
	\begin{lemma}\label{emb}
		Let $\Omega\subset\mathbb{G}$ be a bounded domain with piecewise smooth and simple boundary and assume $1<p<\nu$. Then $W_0^{1,p}(\Omega)$ is continuously embedded in $L^{q}(\Omega)$ for every $q\in[1,\nu^*],$ where $\nu^*=\frac{\nu p}{\nu-p}$. Moreover, the embedding is compact for every $1\leq q<\nu^*$.
	\end{lemma}
	\noindent The following proposition, due to {\sc Ruzhansky-Suragan} \cite{ruzh-suru-1}, will be used on a regular basis. It is an analogue of the divergence theorem in the Euclidean setup.
	\begin{proposition}\label{div_thm}
		Let $f_n\in C^1(\Omega)\cap C(\bar{\Omega})$, $n=1,2,\cdots, N_1$. Then for each $n=1,2,\cdots, N_1$, we have 
		$$\int_{\Omega}X_nf_nd\nu=\int_{\partial\Omega}f_n\langle X_n,d\nu\rangle.$$
		Consequently,
		$$\int_{\Omega}\sum_{n=1}^{N_1}X_nf_nd\nu=\int_{\partial\Omega}\sum_{n=1}^{N_1}f_n\langle X_n,d\nu\rangle.$$
	\end{proposition}
	%\subsection{Necessary tools to tackle the problem}
	\noindent
	Throughout the paper we shall assume that $\mathcal{H}(\Omega)<\infty$. We define an energy functional associated to problem  \eqref{probmain} as follows
	\begin{align*}
		\begin{split}
			I(u)&=\int_{\Omega}\frac{|\nabla_{\mathbb{G}} u|^{2}}{2}dx+\int_{\Omega}\chi_{\{u>1\}}(x)dx-\frac{\lambda}{3}\int_{\Omega}(u-1)_+^3fdx.
		\end{split}
	\end{align*}
	The functional $I$ exhibits the mountain pass geometry. Let $$\Lambda:=\{\psi\in C([0,1];W_0^{1,2}(\Omega)):\psi(0)=0, I(\psi(1))<0\}$$ which consists of paths joining $u=0$ and the set of points $\{u\in W_0^{1,2}(\Omega):I(u)<0\}$. We further define $$c:=\underset{\psi\in\Lambda}\inf~\underset{u\in \psi([0,1])}\max I(u).$$ 
	However, this functional is not even differentiable and hence is an ineligible candidate to fit into the realm of the variational setup. We first define a smooth function $g:\mathbb{R}\rightarrow [0,2]$ as follows
	\[g(t)= \begin{cases}\label{smooth_func_1}
		0, & \text{if}~t\leq 0 \\
		\text{a positive function}, & \text{if}~0<t<1\\
		0, & \text{if}~t\geq 1
	\end{cases}\]
	and $\int_0^1g(t)dt=1$.
	We further let $G(t)=\int_0^tg(t)dt$. Clearly, $G$ is smooth and nondecreasing function such that 
	\[G(t)= \begin{cases}\label{smooth_func_2}
		0, & \text{if}~t\leq 0 \\
		\text{a positive function}<1, & \text{if}~0<t<1\\
		1, & \text{if}~t\geq 1.
	\end{cases}\]
%	We set $$h_{\alpha}(x,t)=G\left(\frac{t}{\alpha}\right)t^2f(x),
%	\
%	H_{\alpha}(x,t)=\int_0^th_{\alpha}(s)ds,
%	\
%	t\geq 0.$$
	Finally,  inspired by the work of {\sc Jerison-Perera} \cite{J_P_1},
	we approximate $I$ using the following functionals 
	which vary with respect to a parameter, $\alpha > 0$, 
	\begin{align*}
		\begin{split}
			I_{\alpha}(u)&=\int_{\Omega}\frac{|\nabla_{\mathbb{G}} u|^{2}}{2}dx+\int_{\Omega}G\left(\frac{u-1}{\alpha}\right)dx-\frac{\lambda}{3}\int_{\Omega}(u-1)_+^2fdx.
		\end{split}
	\end{align*}
	%We note that this functional $I_{\alpha}$, is of at least $C^2$ class and therefore
	%\begin{align*}
	%\langle I_{\alpha}''(u)v,w\rangle=&\int_{\Omega}\tilde{\nabla} vwdx+\int_{\Omega}\frac{1}{\alpha^2}g'\left(\frac{u-1}{\alpha}\right)vwdx.
	%\end{align*}
%	\begin{definition}[see Definition $1$ \cite{hofer}]\label{mp_point}
%	$u$ is said to be a point of mountain-pass type associated with the functional $J$ if for each open neighbourhoods $\mathcal{U}$ of $u$, $\{u:J(u)<d\}\cap\mathcal{U}$ is nonempty and not path connected.
%	\end{definition}
\noindent An essential condition in variational techniques which a functional $J:X\to\mathbb{R}$ requires to satisfy is the Palais-Smale (PS) condition. It states that if $J(w_n)\to c$ and $J'(u_n)\to 0$ in $X^*$, the dual of $X$, then there exists a subsequence of $(w_n)$ which strongly converges to, say $w$, in $X$. We shall prove that both  functionals $I, I_{\alpha}$ defined above satisfy the (PS) condtion.
%\\
%Incidentally, both the functionals $I, I_{\alpha}$ defined above satisfies the (PS) condtion which can be shown by standard arguments. 
	\section{Monotonicity Lemma}\label{s3}
	%This section has been subdivided into two subsections, {\it viz.} the proof of monotonicity result and the proof of the Theorem \ref{convergence1}. 
	%\subsection{Monotonicity result}
	Following the argument in  {\sc Caffarelli et al.} \cite[Theorem $5.1$]{Caffa_jeri_kenig}, we shall  prove an important monotonicity result  stated below.
	We refer to the monograph by {\sc Nagel} \cite[ Section $1.2$]{Nagel1} for the background regarding our proof in a non-Euclidean setup.  Most of our modifications are required by the differences from the Euclidean setting.	
	
	\begin{lemma}\label{conv_res1}
		Let $u>0$ be a Lipschitz continuous function on the unit ball $B_1(0)\subset\mathbb{G},$ satisfying the distributional inequalities
		\begin{align}\label{mono_ineq1}\pm\mathcal{L} u\leq \left(\dfrac{\lambda}{\alpha}\chi_{\{|u-1|<\alpha\}}(x)\mathcal{F}(|\nabla_{\mathbb{G}}u|)+A\right),\end{align}
		for constants $A>0$, $0<\alpha\leq 1.$
		Suppose further that $\mathcal{F}$ is
		a continuous function such that $\mathcal{F}(t)=o(t^2)$ near infinity. Then there exist $C=C(N,A)>0$ and $\int_{{B_1}(0)}u^2dx$, but not on $\alpha$, such that 
		$$\underset{x\in B_{\frac{1}{2}}(0)}{\rm{esssup}}\{|\nabla_{\mathbb{G}} u(x)|\}\leq C.$$
	\end{lemma}
	\begin{proof}
		Let $u$ be a Lipschitz continuous function on the unit ball $B_1(0)\subset\mathbb{G}.$
		Denote $$v(x)=\frac{15}{\alpha}u\left(\frac{\alpha}{15} x\right),
		\quad 
		v_1=v+\underset{B_{1/4}}{\max}\{v_-\}.$$
		\noindent Since the proof is quite technical in nature, before giving the proof we sketch the idea. The primary challenge is to prove that $|\nabla v|$ is bounded on, say $B_{1/32}$. In Step $1$ we shall establish the $L^{\infty}$ bound on $v_1$, where $v_1$ is a perturbation of $v$. Next, we shall show in Step $2$ that a uniform bound on $|\nabla v|$ exists and this depends on the bound on $v$ and \eqref{mono_ineq1}. This step is also essential to establish an interior regularity estimate for the semilinear equation independent of the monotonicity theorem. The monotonicity theorem also helps to produce an $L^{\infty}$ bound on $v$. A meticulous choice of $\beta>0$ has to be made so that $\mathcal{F}(t)\leq \beta t^2+A(\beta)$.\\
		{\it Step $1$:}~Since $u$ is a Lipschitz continuous function on the unit ball,  it is also bounded on it  by a constant say, $M_0$. By {\sc Magnani-Rajala} \cite[Theorem $1.1$]{rajala1}, $u$ is also differentiable a.e. on $B_1(0)$. Therefore, $0\leq v_1\leq M_1$.\\
%		We shall prove the result stated in the lemma only for $u_+$, since the proof for $u_-$ will follow suit.
\noindent	{\it Step $2$:}~Let us choose a function $\eta\in C_0^{\infty}(B_{1/4})$  such that $0\leq\eta\leq 1$ in $B_{3/4}$ and $\eta=1$ in $B_{1/2}$. Furthermore, for any $\beta\in(0,1]$ we have a positive finite number $A(\beta)$ such that 
		\begin{align}\label{est_0}
		\mathcal{F}(t)&\leq A(\beta)+\beta t^2.
		\end{align}
		Thus by testing with $\eta^2v_1,$ we have
		\begin{align}\label{estimate1}
			\begin{split}
				\int_{\Omega}\eta^2|\nabla_{\mathbb{G}} v_1|^2=&-\int_{\Omega}(2v_1\eta(\tilde{\nabla}v_1 \eta)+\eta^2v_1\mathcal{L} v_1dx)dx\\
				\leq&\frac{1}{2}\int_{\Omega}\eta^2|\nabla_{\mathbb{G}} v_1|^2dx+2\int_{\Omega}v_1^2|\nabla_{\mathbb{G}} \eta|^2dx
				+AM_1\int_{\Omega}\eta^2\left(A(\beta)+\beta|\nabla_{\mathbb{G}} v_1|^2\right)dx\\
				\leq&\frac{1}{2}\int_{\Omega}\eta^2|\nabla_{\mathbb{G}} v_1|^2dx+pM_1^2\int_{\Omega}|\nabla_{\mathbb{G}} \eta|^2dx
				+M_1\int_{\Omega}\eta^2\left(\beta|\nabla_{\mathbb{G}} v_1|^2+A(\beta)\right)dx.
			\end{split}
		\end{align}
	Here, $\tilde{\nabla}v_1\eta=\sum_{k=1}^{N_1}X_kv_1X_k$. It is thus established that 
		\begin{align}\label{estimate2}
			\frac{1}{2}\int_{B_{1/2}}|\nabla_{\mathbb{G}} v_1|^2dx&\leq M_2.
		\end{align}
		%However, $u$ being Lipschitz continuous, the gradient $\nabla u$ is bounded a.e. in $B_1(0)$ and hence in $B_{1/2}(0)$. Thus $\underset{B_{1/2}(0)}{\text{esssup}}\{|\nabla u|\}\leq C$, for some $C>0$.
		We define the maximal operator by
		\begin{align}\label{max_oper}
			\mathfrak{M}f(x)&=\underset{0<r<1/100}{\sup}\frac{1}{|B_r(x)|}\int_{B_r(x)}f(y)dy.
		\end{align}
		For $\mu>0$, we further denote $$S_{\mu}=\{x\in B_{1/32}:\mathfrak{M}(|\nabla_{\mathbb{G}} v_1|^2)(x)>\mu\}.$$
		%We now claim the following: 
		\begin{Claim}
			There exists
			a constant $C_1$ such that for any $\epsilon>0$ there exists a finite positive number $\mu_0$ such that for any $\mu\geq\mu_0,$
			\begin{enumerate}
				\item $|S_{\mu}\cap Q_0|\leq |S_{\mu_0}\cap Q_0|<\epsilon|Q_0|,$ where $Q_0$ is a cube with side length $2^{-10-10N}$ and $Q_0\cap B_{1/32}\neq\emptyset$.
				\item If $Q$ is a dyadic subcube of $Q_0$ for which $|S_{C_1\mu}\cap Q|\geq \epsilon|Q|$, then $Q\subset Q^*\subset S_{\mu}$, where $Q$ is an immediate dyadic subcube of $Q^*$.
			\end{enumerate}
		\end{Claim}
		\begin{proof}
			%{\it Proof of the claim}:~
			We only sketch the proof of the claim as the ideas are borrowed from \cite{Caffa_jeri_kenig}. Assertion (1) follows from the argument given in \cite{Caffa_jeri_kenig}.\\
			Suppose now that Assertion (2) fails to hold. Then one can find a cube $Q$ such that $|S_{\mu}\cap Q|\geq\epsilon|Q|$ and $y\in Q^*$, however $\mathfrak{M}(|\nabla_{\mathbb{G}} v_1|^2)(y)\leq \mu$. Let $\rho$ be $2^{6N}$ times the length of the side of $Q$ and consider $\mathfrak{M}_{\rho/4}(|\nabla_{\mathbb{G}} v_1|^2)(0)$, with the supremum taken over $(0,\rho/4)$. Since $\mathfrak{M}(|\nabla_{\mathbb{G}} v_1|^2)(y)\leq \mu$, there exists a constant $C_2$ such that for any $x\in Q$,
			\begin{align}\label{est1}
				\mathfrak{M}(|\nabla_{\mathbb{G}} v_1|^2)(x)&\leq\max\{\mathfrak{M}_{\rho/4}(|\nabla_{\mathbb{G}} v_1|^2)(x),C_2\mu\}.
			\end{align}
			Let $\phi$ be such that
			\begin{align}\label{harm_func}
				\begin{split}
					-\mathcal{L}\phi&=0~\text{in}~B_{\rho}(y)\\
					\phi&=v_1~\text{on}~\partial B_{\rho}(y).
				\end{split}
			\end{align}
			Since $\phi$ is a minimizer of the functional $\frac{1}{2}\int_{B_{\rho}(y)}|\nabla_{\mathbb{G}}\phi|^2 dx$, we have 
			\begin{align}\label{est2}
				\int_{B_{\rho}(y)}|\nabla_{\mathbb{G}}\phi|^2 dx&\leq \int_{B_{\rho}(y)}|\nabla_{\mathbb{G}} v_1|^2 dx\leq\mu|B_{\rho}(y)|.
			\end{align}
			Of course, we have the mean value property at our disposal (see  {\sc Adamowicz-Warhurst} \cite[Condition 1]{adam1})  to establish that 
			\begin{align}\label{est3}
				\underset{B_{\rho/2}(y)}{\sup}\{|\nabla_{\mathbb{G}} \phi|^2\}&\leq C_3\mu.
			\end{align}
			On choosing $C_1=15\max\{C_2,C_3\}$ we have
			\begin{align}\label{est4}
				\mathcal{A}:=\{x\in Q:\mathfrak{M}_{\rho/4}(|\nabla_{\mathbb{G}} v_1|^2)(x)>C_1\mu\}=\{x\in Q:\mathfrak{M}(|\nabla_{\mathbb{G}} v_1|^2)(x)>C_1\mu\}=:\mathcal{B}.
			\end{align}
			If $x\in\mathcal{A}$, then it is easy see that $x\in\mathcal{B}$. Thus $\mathcal{A}\subset\mathcal{B}$. Suppose that $x\in\mathcal{B}$. Then $\mathfrak{M}(|\nabla_{\mathbb{G}} v_1|^2)(x)>C_1\mu$. However, by \eqref{est1} and by the choice of $C_1$ we have that $x\in\mathcal{A}$.\\
			\noindent Also observe that $$\{x\in Q:\mathfrak{M}_{\rho/4}(|\nabla_{\mathbb{G}} \phi|^2)(x)>C_1\mu/4\}=\emptyset.$$ For if not, then there exists $x\in Q$ such that $\mathfrak{M}(|\nabla_{\mathbb{G}} \phi|^2)(x)>C_\mu/4$. One can thus produce $r\in (0,\rho/4)$ such that $$\frac{C_1\mu}{4}<\frac{1}{|B_r(x)|}\int_{B_r(x)}|\nabla_{\mathbb{G}} \phi|^2dy\leq\frac{C_1\mu}{15}.$$
			This is a contradiction since this leads to an absurdity $4>15$.
			Therefore,  
			\begin{align}\label{est5}
				\begin{split}
					|\{x\in Q&: \mathfrak{M}_{\rho/4}(|\nabla_{\mathbb{G}} v_1|^2)>C_1\mu\}|\\
					&\leq|\{x\in Q: \mathfrak{M}_{\rho/4}(|\nabla_{\mathbb{G}} (v_1-\phi)|^2)+\mathfrak{M}_{\rho/4}(|\nabla_{\mathbb{G}}\phi|^2)>C_1\mu/2\}|\\
					&\leq|\{x\in Q: \mathfrak{M}_{\rho/4}(|\nabla_{\mathbb{G}} (v_1-\phi)|^2)>C_1\mu/4\}|+|\{\mathfrak{M}_{\rho/4}(|\nabla_{\mathbb{G}}\phi|^2)>C_1\mu/4\}|\\
					&=|\{x\in Q: \mathfrak{M}_{\rho/4}(|\nabla_{\mathbb{G}}(v_1-\phi)|^2)>C_1\mu/4\}|.
				\end{split}
			\end{align}
			Thus there exists a constant $C_4$, which follows by the weak $(1,1)$ inequality for
			 $\mathfrak{M}$, such that
			\begin{align}\label{est6}
				C_4\mu^{-1}\int_{B_{\rho}(y)}|\nabla_{\mathbb{G}}(v_1-\phi)|^2dx&\geq|\{x\in Q:\mathfrak{M}_{\rho/4}(|\nabla_{\mathbb{G}}(v_1-\phi)|^2)>C_1\mu/4\}|.
			\end{align}
			Furthermore, by the maximum principle we have $|v_1-\phi|\leq C$ on the ball $B_{\rho}(y)$. By the weak formulation of  problem \eqref{harm_func}, we have 
			\begin{align}\label{est7} 
				0&=\int_{B_{\rho}(y)}\tilde{\nabla}\phi(v_1-\phi)dx.
			\end{align}
			Furthermore,
			\begin{align}\label{est8}
				\begin{split}
					-\int_{B_{\rho}(y)}\mathcal{L}v_1(v_1-\phi)dx&=-\int_{B_{\rho}(y)}(\mathcal{L}v_1-\mathcal{L}\phi)(v_1-\phi)dx.
				\end{split}
			\end{align}
			Thus we have 
			\begin{align}\label{est8'}
				\begin{split}
					C_5\int_{B_{\rho}(y)}|\nabla_{\mathbb{G}}(v_1-\phi)|^2dx&=\int_{B_{\rho}(y)}\tilde{\nabla}(v_1-\phi)(v_1-\phi)dx\\
					&\leq -\int_{B_{\rho}(y)}(\mathcal{L}v_1-\mathcal{L}\phi)(v_1-\phi)dx
					=-\int_{B_{\rho}(y)}(\mathcal{L}v_1)(v_1-\phi)dx\\
					&\leq\int_{B_{\rho}(y)}C\left(\beta|v_1|^2+A(\beta)\right)dx.
				\end{split}
			\end{align}
			Using  inequality \eqref{est6}, we get 
			\begin{align}\label{est9}
				\begin{split}
					|\{x\in Q:\mathfrak{M}_{\rho/4}(|\nabla_{\mathbb{G}} v_1|^2)>C_1\mu\}|&\leq C_6\left(\beta+\frac{A(\beta)}{\mu}\right)|Q|.
				\end{split}
			\end{align}
			Thus, for a sufficiently small $\delta>0$ and large $\mu>0,$ we have $$C_6\delta<\epsilon/3
			\quad
			C_6A(\beta)/\mu<\epsilon/3.$$ Therefore
			$$\{x\in Q: \mathfrak{M}(|\nabla_{\mathbb{G}} v_1|^2)>C_1\mu\}<\epsilon|Q|,$$
			which indeed is a contradiction to the hypothesis. Therefore, assertion (2) indeed holds.
		\end{proof}
		\noindent One can now follow verbatim   \cite{Caffa_jeri_kenig} to conclude that assertion (2) leads to 
		\begin{align}\label{est10}
			\begin{split}
				|S_{C_1^k\mu}\cap Q_0|&\leq\epsilon^{k+1}|Q_0|.
			\end{split}
		\end{align}
		We further note from \eqref{est10} that for any $1<\theta<\infty,$ a sufficiently small $\epsilon>0$ can be chosen so that $\mathfrak{M}(|\nabla_{\mathbb{G}} v_1|^2)$ is bounded in $L^{\theta}(B_{1/16})$, i.e.
		\begin{align}\label{est11}
			\begin{split}
				\int_{B_{1/16}}|\nabla_{\mathbb{G}} v_1|^{\theta}dx&\leq C_{7},
			\end{split}
		\end{align}
		where $C_7$ is a uniform constant that depends on $\theta, A, \mathcal{F}$. On choosing $\theta=N,$ we have $2\theta>N$. Hence we obtain  from \eqref{est_0}
		\begin{align}\label{est12}
			\begin{split}
				\underset{B_{1/32}}\sup\{|\nabla_{\mathbb{G}} v_1|\}&\leq C_8.
			\end{split}
		\end{align}
		Reverting back to the variables in terms of $u,$ we get 
		\begin{align}\label{est13}
			\begin{split}
				\underset{B_{\alpha/320}(x)}\sup\{|\nabla_{\mathbb{G}} u|\}&\leq C_8~\text{for any}~x\in B_{1/4}~\text{such that}~|\{u(x)<\alpha\}|,
			\end{split}
		\end{align}
		and in order to finally arrive at the conclusion 
		\begin{align}\label{est14}
			\begin{split}
				\underset{B_{r/4}(0)}\sup\{|\nabla_{\mathbb{G}} u|\}&\leq C_9,
			\end{split}
		\end{align}
		we follow the proof of \cite{Caffa_jeri_kenig} again, however with the choice of $$w(x)=A_0r(r^{N-2}|x|^{2-N}-1)+A(|x|^2-r^2)+O(\alpha).$$ Therefore $\underset{B_{r/2}}{\sup}\{|\nabla_{\mathbb{G}} u|\}<\infty$.
	\end{proof}
\begin{remark}\label{mon_req}
The above monotonicity bound of the type \eqref{mono_ineq1} implies uniform Lipschitz continuity of a family of solutions to a class of semilinear equations with free boundary conditions. In fact, a very important component in the passage to the limit
 in the proof of Theorem \ref{main_res2} in Section~\ref{s5}
 will be the uniform Lipschitz continuity result derived
 % from Lemma \ref{conv_res1}
  in the next section.
\end{remark}
%\begin{remark}\label{sketch}
%...
%\end{remark}
	\section{Convergence Lemma}\label{s4}
	Before proving  Theorem \ref{main_res2}, we shall prove the following convergence result, which is also of independent interest. It  helps us to conclude that the obtained solution is nontrivial in the sense of Remark \ref{nontrivial}.	
\begin{lemma}\label{convergence1}
	Let $(\alpha_j)$ be a sequence of positive numbers such that $\alpha_j\rightarrow 0,$ as $j\rightarrow\infty,$ and let $u_j$ be a critical point of $I_{\alpha_j}$. 
	Suppose that $(u_j)$ is bounded in $W_0^{1,2}(\Omega)\cap L^{\infty}(\Omega).$
	Then there exists a Lipschitz continuous function $u$ on $\bar{\Omega}$ such that $u\in W_0^{1,2}(\Omega)\cap C^2(\bar{\Omega}\setminus H(u))$, and for a renamed subsequence  the following holds:
		\begin{enumerate}[label=(\roman*)]
			\item $u_j\rightarrow u$ uniformly over $\bar{\Omega};$
			\item $u_j\rightarrow u$ locally in $C^1(\bar{\Omega}\setminus\{u=1\});$
			\item $u_j\rightarrow u$ strongly in $W_0^{1,2}(\Omega);$ and
			\item $I(u)\leq\lim\inf I_{\alpha_j}(u_j)\leq\lim\sup I_{\alpha_j}(u_j)\leq I(u)+|\{u=1\}|$.
	\end{enumerate}
	In other words, $u$ is a nontrivial function if $\lim\inf I_{\alpha_j}(u_j)<0$ or $\lim\sup I_{\alpha_j}(u_j)>0$.
	Furthermore, $u$ satisfies $-\mathcal{L}u(x)=\lambda(u(x)-1)_+^2f(x)$ classically in $\Omega\setminus H(u)$ and $u$ satisfies the free boundary condition in the generalized sense and  vanishes continuously on $\partial\Omega$. In the case when  $u$ is  nontrivial, the set $\{u>1\}$ is nonempty.
\end{lemma}
	\begin{proof}
	Let $0<\alpha_j<1$. Consider the sequence of problems $(P_j)$
	\begin{align}\label{approx_prob_1}
		\begin{split}
			-\mathcal{L} u_j&=-\frac{1}{\alpha_j}g\left(\frac{(u_j-1)_+}{\alpha_j}\right)+\lambda (u-1)_+^2f~\text{in}~\Omega,\\
			u_j&>0~\text{in}~\Omega,\\
			u_j&=0~\text{on}~\partial\Omega.
		\end{split}
	\end{align}
	The nature of the problem allows us to conclude by an iterative technique that the sequence $(u_j)$ is bounded in $L^{\infty}(\Omega)$. Hence, there exists $C_0$ such that $0\leq (u_j-1)_+^2f\leq C_0$. Let $\varphi_0$ be a solution of 
	\begin{align}\label{approx_prob_2}
		\begin{split}
			-\mathcal{L} \varphi_0&=\lambda C_0~\text{in}~\Omega\\
			\varphi_0&=0~\text{on}~\partial\Omega.
		\end{split}
	\end{align}
	Now, since $g\geq 0$, we have that $-\mathcal{L} u_j\leq \lambda C_0=\mathcal{L} \varphi_0$ in $\Omega$. Therefore by the maximum principle, 
	\begin{align}\label{comparison1}
		0\leq u_j(x)\leq \varphi_0(x)~\text{for all} ~x\in\Omega.
	\end{align}
	Since $\{u_j\geq 1\}\subset\{\varphi_0\geq 1\}$, it follows that $\varphi_0$ gives a uniform lower bound, say $d_0$, on the distance from the set $\{u_j\geq 1\}$ to $\partial\Omega$. Thus $(u_j)$ is bounded with respect to the $C^{2,a}$ norm. Therefore it has a convergent subsequence in the $C^2$-norm on  $\dfrac{d_0}{2}$-neighbourhood of the boundary $\partial\Omega$. Obviously, $0\leq g\leq 2\chi_{(-1,1)}$ and hence
	\begin{align}\label{comparison2}
		\begin{split}
			\pm\mathcal{L} u_j&=\pm\frac{1}{\alpha_j}g\left(\frac{(u_j-1)_+}{\alpha_j}\right)\mp\lambda (u_j-1)_+^2f\\
			&\leq\frac{2}{\alpha_j}\chi_{\{|u_j-1|<\alpha_j\}}(x)+\lambda C_0.
		\end{split}
	\end{align}
	Since, $(u_j)$ is bounded in $L^2(\Omega)$, there exists by Lemma \ref{conv_res1}, $A>0$ such that 
	\begin{align}\label{aux1}
		\underset{x\in B_{\frac{r}{2}}(x_0)}{\text{esssup}}\{|\nabla_{\mathbb{G}} u_j(x)|\}&\leq\frac{A}{r},
	\end{align}
	for a suitable $r>0$ for which $B_r(0)\subset\Omega$. However, since $(u_j)$ is a sequence of Lipschitz continuous functions which incidentally are also $C^1$ functions a.e., we have
	\begin{align}\label{aux2}
		\underset{x\in B_{\frac{r}{2}}(x_0)}{\sup}\{|\nabla_{\mathbb{G}} u_j(x)|\}&\leq\frac{A}{r}.
	\end{align}
	Therefore $(u_j)$ is a sequence of uniformly Lipschitz continuous functions on the compact subsets, say $K$, of $\Omega$ such that $d(K,\partial\Omega)\geq\frac{d_0}{2}$. By the Ascoli-Arzela theorem  applied to $(u_j),$ we get  a subsequence, still referred to by the same name, that  converges uniformly to a Lipschitz continuous function $u$ in $\Omega$ which vanishes on the boundary $\partial\Omega$. The convergence is strong in $C^2$ on a $\frac{d_0}{2}$-neighbourhood of $\partial\Omega$. By the  Banach-Alaoglu theorem we can conclude that $u_j\rightharpoonup u$ in $W_0^{1,2}(\Omega)$.\\
	We now prove that $u$ satisfies
	\begin{align}\label{aux_prob_sat}
		-\mathcal{L} u&=\lambda(u-1)_+^2f
	\end{align}
	on the set $\{u\neq 1\}$. Let $\varphi\in C_0^{\infty}(\{u>1\})$. Thus $u\geq 1+2\delta$ on the support of $\varphi$ for some $\delta>0$. By using the convergence of $u_j$ to $u$ uniformly on $\Omega,$ we conclude that $|u_j-u|<\delta$. Thus for any sufficiently large $j$ with $\delta_j<\delta$ we have $u_j\geq 1+\delta_j$ on the support of $\varphi$. Testing \eqref{aux_prob_sat} with $\varphi$ yields
	\begin{align}\label{weak_conv_1}
		\int_{\Omega}\tilde{\nabla}u_j\varphi dx&=\lambda\int_{\Omega}(u_j-1)_+^2f\varphi dx.
	\end{align}
	By passing the limit $j\rightarrow\infty$ to \eqref{aux_prob_sat}, we obtain
	\begin{align}\label{weak_conv_2}
		\int_{\Omega}\tilde{\nabla}u\varphi  dx&=\lambda\int_{\Omega}(u-1)_+^2f\varphi dx.
	\end{align}
	In order to obtain \eqref{weak_conv_2} we have used the weak and uniform convergence of $u_j$ to $u$ in $W_0^{1,2}(\Omega)$ and $\Omega,$ respectively. Therefore $u$ is a weak solution of $-\mathcal{L}u=\lambda f$ in $\{u>1\}$. 
	%	Since $u$ is a Lipschitz continuous function, hence by the Schauder estimates we conclude that it is also a classical solution of $-\Delta_p u=\lambda (u-1)_+^{q-1}$ in $\{u>1\}$. 
	Similarly, by choosing $\varphi\in C_0^{\infty}(\{u<1\}),$ we can similarly find a $\delta>0$ such that $u\leq 1-2\delta$ due to which $u_j<1-\delta$.\\
	\noindent We now analyze the nature of $u$ on the set $\{u\leq 1\}^{\circ}$. Testing \eqref{aux_prob_sat} with any nonnegative function, passing to the limit $j\rightarrow\infty$ and using the fact that $g\geq 0$, $G\leq 1,$ it can be shown that $u$ satisfies 
	\begin{align}\label{weak_conv_3}
		\mathcal{L} u&\leq\lambda(u-1)_+^2f~\text{in}~\Omega
	\end{align}
	in the sense of distribution. Furthermore, $\mu=\mathcal{L}(u-1)_{-}$ is a positive Radon measure supported on $\Omega\cap\partial\{u<1\}$ (the reader can refer to Lemma \ref{positive_Rad_meas} in Section 5). From \eqref{weak_conv_3}, $\mu>0$ and the usage of the regularity result by {\sc Gilbarg-Trudinger} \cite[Section $9.4$]{Gil_trud} we establish that $u\in W_{\text{loc}}^{2,2}(\{u\leq 1\}^{\circ})$. Hence $\mathcal{M}$ is supported on $\Omega\cap\partial\{u<1\}\cap\partial\{u>1\}$ and $u$ satisfies $\mathcal{L} u=0$ on the set $\{u\leq 1\}^{\circ}$.\\
	\noindent To prove $(ii)$, we
	shall
	show that $u_j\rightarrow u$ locally in $C^1(\Omega\setminus\{u=1\})$. We
	have
	already proved that $u_j\rightarrow u$ with respect to the $C^2$ norm in a neighbourhood of $\partial\Omega$ of $\bar{\Omega}$. Let $M\subset\subset\{u>1\}$. In this set $M$ we have $u\geq 1+2\delta$ for some $\delta>0$. Hence, for sufficiently large $j$, with $\delta_j<\delta$, we have $|u_j-u|<\delta$ in $\Omega$ and hence $u_j\geq 1+\delta_j$ in $M$. From \eqref{approx_prob_1} we have $$\mathcal{L} u_j=\lambda (u-1)^2f~\text{in}~M.$$
	%	Clearly, $(u_j-1)_+^{q-1}\rightarrow (u-1)_+^{q-1}$ in $L^p(\Omega)$ for $1<p<\infty$ and $u_j\rightarrow u$ uniformly in $\Omega$. 
	This analysis says something more stronger - {\it since $\mathcal{L}u_j=\lambda (u-1)^2f$ in $M$, we have that $u_j\rightarrow u$ in $W^{2,2}(M)$}. By the embedding $W^{2,2}(M)\hookrightarrow C^1(M)$ for $p>2$, we have $u_j\rightarrow u$ in $C^1(M)$. This proves that $u_j\rightarrow u$ in $C^1(\{u>1\})$. Similarly, we can also show that $u_j\rightarrow u$ in $C^1(\{u<1\})$.\\
	\noindent We shall now prove $(iii)$. Since $u_j\rightharpoonup u$ in $W_0^{1,p}(\Omega)$, we have by the weak lower semicontinuity of the norm $\|\cdot\|,$ 
	$$\|u\|\leq\lim\inf\|u_j\|.$$
	It suffices
	to prove that $\lim\sup\|u_j\|\leq \|u\|$. To this end, we multiply \eqref{approx_prob_1} with $(u_j-1)$ and then integrate by parts. We shall also use that $tg\left(\frac{t}{\delta_j}\right)\geq 0$ for any $t\in\mathbb{R}$. This yields
	\begin{align}\label{weak_conv_4}
		\begin{split}
			\int_{\Omega}|\nabla_{\mathbb{G}} u_j|^2dx&\leq \lambda\int_{\Omega}(u_j-1)_+^2fdx-\int_{\partial\Omega}u_j\langle X_i,dn\rangle dS\\
			&\rightarrow\lambda\int_{\Omega}(u-1)_+^2fdx-\int_{\partial\Omega}u\langle X_i,dn\rangle dS
		\end{split}
	\end{align}
	as $j\rightarrow\infty$. 
	\noindent We choose $\vec{\varphi}\in C_0^1(\Omega,\mathbb{G})$ such that $u\neq 1$ a.e. on the support of $\vec{\varphi}$. Multiplying by $\sum_{k=1}^{N}\varphi_kX_ku_n$ 
	 the weak formulation of \eqref{approx_prob_1} and integrating over the set $\{1-\epsilon^-<u_n<1+\epsilon^+\},$ we get
	\begin{align}\label{weak_conv_10}
		\begin{split}
			\int_{\{1-\epsilon^-<u_n<1+\epsilon^+\}}\left[-\Delta_{\mathbb{G}} u_n+\frac{1}{\alpha_n}g\left(\frac{u_n-1}{\alpha_n}\right)\right]\sum_{k=1}^{N}\varphi_kX_ku_ndx\\
			=\lambda\int_{\{1-\epsilon^-<u_n<1+\epsilon^+\}}(u_n-1)_+^2f\sum_{k=1}^{N}\varphi_kX_ku_n dx.
		\end{split}
	\end{align}
	The term on the left hand side of \eqref{weak_conv_10} can be expressed as follows:
	\begin{align}\label{weak_conv_11}
		\begin{split}
			\nabla_{\mathbb{G}}\cdot\left(\frac{1}{2}|\nabla_{\mathbb{G}} u_n|^2\vec{\varphi}-(\sum_{k=1}^{N}X_ku_n\varphi_k )\nabla_{\mathbb{G}} u_n\right)&+\sum_{k=1}^{N}\sum_{l=1}^{N}X_l\varphi_kX_lu_nX_ku_n\\
			&-\frac{1}{2}|\nabla_{\mathbb{G}} u_n|^2\nabla_{\mathbb{G}}\cdot\vec{\varphi}+\sum_{k=1}^{N}\varphi_kX_kG\left(\frac{u_n-1}{\alpha_n}\right).
		\end{split}
	\end{align}
	Using \eqref{weak_conv_11} and on integrating by parts, we obtain
	\begin{align}\label{weak_conv_12}
		\begin{split}
			\int_{\{u_n=1+\epsilon^+\}\cup\{u_n=1-\epsilon^-\}}\left[\frac{1}{2}|\nabla_{\mathbb{G}}u_n|^2\sum_{k=1}^{N}\varphi_k\langle X_k,dn\rangle-(\sum_{k=1}^{N}X_ku_n\varphi_k )\sum_{l=1}^{N}X_lu_n\langle X_l,dn\rangle\right.\\
			\left.+G\left(\frac{u_n-1}{\alpha_j}\right)\sum_{k=1}^{N}\varphi_k\langle X_k,dn\rangle\right]\\
			=\int_{\{1-\epsilon^-<u_n<1+\epsilon^+\}}\left(\frac{1}{2}|\nabla_{\mathbb{G}} u_n|^2\sum_{k=1}^{N}X_k\varphi_k-\sum_{k=1}^{N}\sum_{l=1}^{N}X_k\varphi_lX_lu_nX_ku_n\right)dx\\
			+\int_{\{1-\epsilon^-<u_n<1+\epsilon^+\}}\left[G\left(\frac{u_n-1}{\alpha_n}\right)\sum_{k=1}^{N}X_k\varphi_k+\lambda (u_n-1)_+^2f\sum_{k=1}^{N}X_k\varphi_k\right]dx.
		\end{split}
	\end{align}
	The integral on the left of equation \eqref{weak_conv_12} converges to 	
	\begin{align}\label{weak_conv_12'}
		\begin{split}
			&\int_{\{u=1+\epsilon^+\}\cup\{u=1-\epsilon^-\}}\left[\frac{1}{2}|\nabla_{\mathbb{G}}u|^2\sum_{k=1}^{N}\varphi_k\langle X_k,dn\rangle-(\sum_{k=1}^{N}X_ku\varphi_k )\sum_{l=1}^{N}X_lu\langle X_l,dn\rangle\right.\\
			&\left.+\int_{\{u=1+\epsilon^+\}}\sum_{k=1}^{N}\varphi_k\langle X_k,dn\rangle\right]\\
			=&\int_{\{u=1+\epsilon^+\}\cup\{u=1-\epsilon^-\}}\left[\left(1-\frac{1}{2}|\nabla_{\mathbb{G}}u|^2\right)\sum_{k=1}^{N}\varphi_k\langle X_k,dn\rangle-\sum_{k\neq l;1\leq k,l\leq N}\varphi_kX_luX_ku\langle X_l,dn\rangle\right]\\
			=&\int_{\{u=1+\epsilon^+\}}\left[\left(1-\frac{1}{2}|\nabla_{\mathbb{G}}u|^2\right)\sum_{k=1}^{N}\varphi_k\langle X_k,dn\rangle-\sum_{k\neq l;1\leq k,l\leq N}\varphi_kX_luX_ku\langle X_l,dn\rangle\right]
			\end{split}
			\end{align}
			\begin{align*}
			\begin{split}
			&-\int_{\{u=1-\epsilon^-\}}\left[\left(\frac{1}{2}|\nabla_{\mathbb{G}}u|^2\right)\sum_{k=1}^{N}\varphi_k\langle X_k,dn\rangle-\sum_{k\neq l;1\leq k,l\leq N}\varphi_kX_luX_ku\langle X_l,dn\rangle\right]\\
		=&\int_{\{1-\epsilon^-<u<1+\epsilon^+\}}\left(\frac{1}{2}|\nabla_{\mathbb{G}} u|^2\sum_{k=1}^{N}X_k\varphi_k-\sum_{k=1}^{N}\sum_{l=1}^{N}X_k\varphi_lX_luX_ku\right)dx\\
		&+\int_{\{1-\epsilon^-<u<1+\epsilon^+\}}\left[\sum_{k=1}^{N}X_k\varphi_k+\lambda(u_n-1)_+^2 f\sum_{k=1}^{N}X_k\varphi_k\right]dx.
	\end{split}
	\end{align*}
	as $n\to\infty$.\\
	Note that the normal vector at the point $P$ on the set $\{u=1+\epsilon^{+}\}\cup\{u=1-\epsilon^{-}\}$ is $n=\pm\frac{\nabla_{\mathbb{G}}u(P)}{|\nabla_{\mathbb{G}}u(P)|}$.
	Thus  equation \eqref{weak_conv_12'} under the limit $\epsilon\rightarrow 0$ becomes 
	\begin{align}\label{weak_conv_14}
		\begin{split}
			0=&\underset{\epsilon\rightarrow 0}{\lim}\int_{\{u=1+\epsilon^+\}}\left[\left(1-\frac{1}{2}|\nabla_{\mathbb{G}}u|^2\right)\sum_{k=1}^{N}\varphi_k\langle X_k,dn\rangle\right]\\
			&-\underset{\epsilon\rightarrow 0}{\lim}\int_{\{u=1-\epsilon^-\}}\left[\left(\frac{1}{2}|\nabla_{\mathbb{G}}u|^2\right)\sum_{k=1}^{N}\varphi_k\langle X_k,dn\rangle\right].
		\end{split}
	\end{align}
	This proves that $u$ satisfies the free boundary condition in the sense of {\it  viscosity}. The solution cannot be trivial since $u\in C^1(\{u>1\})$ and it satisfies the free boundary condition.
%%	Thus a solution to \eqref{probmain} exists 
%	which
%	obeys the free boundary condition besides the Dirichlet boundary condition.
	\end{proof}
	\begin{remark}\label{existence1}
	Notice that $I_{\alpha}$ satisfies the Palais-Smale $(PS)$ condition. To prove this, we define $$u_n^+(x):=\max\{u_n(x),0\}, u^++u^-:=(u-1)_++[1-(u-1)_-]=u.$$ 
	Notice that
	\begin{align}\label{PS_cond_pf}
	\begin{split}
I_{\alpha}(u_n)&\geq 2^{-1}\|u_n\|^2-\frac{\lambda}{3}\int_{\Omega}f(u_n^+)^3dx\\
\langle	I_{\alpha}'(u_n),u_n\rangle&\leq \|u_n\|^2-\lambda\int_{\Omega}f(u_n^+)^3dx+\frac{2}{\alpha}|\Omega|.
	\end{split}
\end{align}
Let $c\in\mathbb{R}$ and consider 
\begin{align}\label{PS_cond_pf1}
	\begin{split}
	c+\sigma\|u_n\|+o(1)\geq I_{\alpha}(u_n)-\frac{1}{3}\langle	I_{\alpha}'(u_n),u_n\rangle&\geq 6^{-1}\|u_n\|^2-\frac{2}{\alpha}|\Omega|.
	\end{split}
\end{align}
This implies that $(u_n)$ is bounded in $W_0^{1,2}(\Omega)$. This implies that there exists a subsequence of $(u_n)$ such that $u_n\rightharpoonup u$ in $W_0^{1,2}(\Omega)$, $u_n\to u$ in $L^3(\Omega)$ and $u_n(x)\to u(x)$ a.e. in $\Omega$. Since $\langle	I_{\alpha}'(u_n),v\rangle\to 0$ as $n\to\infty$ we have
\begin{align}\label{PS_cond_pf2}
	\begin{split}
	\underset{n\to\infty}{\lim}\int_{\Omega}\tilde{\nabla}u_nvdx&=	\underset{n\to\infty}{\lim}\left[\int_{\Omega}\frac{1}{\alpha}g\left(\frac{u_n-1}{\alpha}\right)vdx+\lambda\int_{\Omega}(u_n-1)_+^2vfdx\right]~\text{for all}~v\in W_0^{1,2}(\Omega).
	\end{split}
\end{align}
We choose $v=u_n-u$ in \eqref{PS_cond_pf2} to obtain
\begin{align}\label{PS_cond_pf3}
	\begin{split}
		\underset{n\to\infty}{\lim}\int_{\Omega}\tilde{\nabla}u_n(u_n-u)dx&=	\underset{n\to\infty}{\lim}\left[\int_{\Omega}\frac{1}{\alpha}g\left(\frac{u_n-1}{\alpha}\right)(u_n-u)dx+\lambda\int_{\Omega}(u_n-1)_+^2(u_n-u)fdx\right]\\
		&=0.
	\end{split}
\end{align}
This implies that $u_n\to u$ in $W_0^{1,2}(\Omega)$. Hence $I_{\alpha}$ satisfies the (PS) condition.
\end{remark}

\section{Proof of the Main Theorem}\label{s5}

	\noindent Before we prove the existence of a solution to the problem \eqref{probmain}, we develope a few tools which will be used in the proof. We observe that
	$$I_{\alpha}(u)\leq I(u)~\text{in}~W_0^{1,2}(\Omega).$$ 	
	Furthermore, we have
	\begin{align}\label{yang-per1}
	\begin{split}
	I_{\alpha}(u)&\geq\frac{1}{2}\|u\|^2-\frac{\lambda}{3}\int_{\Omega}|u|^3fdx\\
	&\geq \frac{1}{2}\|u\|^2-\frac{C\lambda}{3}\|f\|_{\infty}\|u\|^3 
	\end{split}
	\end{align}
by Lemma \ref{emb}. Therefore, there exists $r_0=r_0(\nu,\lambda,\|f\|_{\infty})>0$ such that 
\begin{align}\label{yang-per2}
\begin{split}
I_{\alpha}(u)&\geq\frac{1}{4}\|u\|^2
\end{split}
\end{align}
for $\|u\|\leq r_0$. Furthermore, for a fixed nonzero $u$ we have $I_{\alpha}(tu)\to-\infty$ as $t\to\infty$ and hence there exists a function $v_0$ such that $I_{\alpha}(v_0)<0=I_{\alpha}(0)$. This indicates that the set  $$\Lambda_{\alpha}:=\{\psi\in C([0,1];W_0^{1,2}(\Omega)):\psi(0)=0, I_{\alpha}(\psi(1))<0\}$$ is nonempty. Hence by the Mountain pass theorem we have 
\begin{align}\label{yang-per3}
\begin{split}
c_{\alpha}&:=\underset{\psi\in\Lambda_{\alpha}}{\inf}~\underset{u\in\psi([0,1])}{\max}I_{\alpha}(u).
\end{split}
\end{align}
By the definition of the set $\Lambda_{\alpha}$ we have $\Lambda\subset\Lambda_{\alpha}$ and
\begin{align}\label{yang-per4}
\begin{split}
c_{\alpha}\leq\underset{u\in\psi([0,1])}{\max}I_{\alpha}(u)\leq\underset{u\in\psi([0,1])}{\max}I(u)
\end{split}
\end{align}
for all $\psi\in\Lambda$. This implies that $c_{\alpha}\leq c$.
\begin{remark}\label{fin_C}
Let $\phi_1$ be the first eigenfunction pertaining to the first eigen value $\lambda_1$ (see Proposition $3.1$ \cite{chen_square}). Notice that
\begin{align}\label{ineq1}
I(t\phi)\to-\infty~\text{as}~t\to\infty.
\end{align}
Thus there exists $t_*>0$ such that $I(t_*\phi_1)<0$. Consider the path which is defined by $\psi(t)=t\phi_1$ for $t\in[0,t_*]$. Then $\psi$ yields a path from $\Lambda$ on which 
\begin{align}\label{ineq2}
	I(t\phi_1)\leq\mathcal{C}:=\underset{t\geq 0}{\sup}\int_{\Omega}\left(\frac{\lambda_1}{2}t^2\phi_1+1\right)dx.
\end{align}
Therefore $c\leq \mathcal{C}$.
\end{remark}
\begin{proof}[Proof of Theorem \ref{main_res2}]
From Remark \ref{fin_C} we conclude that $c_{\alpha}\leq c\leq \mathcal{C}$. Since $I_{\alpha}$ obeys the (PS) condition, it follows that a limit of the (PS) sequence, say $u_{\alpha}$, can be shown to be a critical point of $I_{\alpha}$. Hence we have $I_{\alpha}(u_{\alpha})=c_{\alpha}$.\\
 Now consider a sequence $\alpha_n$ which converges to zero and name $u_{\alpha_n}$ as $u_n$
 and
  $c_{\alpha_n}$ as $c_n$. By Lemma \ref{convergence1} $(i)-(ii)$, we know that a subsequence of $(u_n)$, still denoted by the same name, converges uniformly in $\bar{\Omega}$, locally in $C^1(\bar{\Omega}\setminus\{u=1\}),$ and strongly in $W_0^{1,2}(\Omega)$, to a locally Lipschitz function $u\in W_0^{1,2}(\Omega)\cap C^2(\bar{\Omega}\setminus H(u))$. Moreover, by \eqref{yang-per2} in Remark \ref{existence1} we have $\lim\sup I_{\alpha_n}(u_n)=\lim\sup c_n\geq \frac{r_0}{4}>0$. This indicates that one of the limit conditions $\lim\sup I_{\alpha_n}(u_n)>0$ or $\lim\inf I_{\alpha_n}(u_n)<0$ in Lemma \ref{convergence1} indeed holds. \\ Hence by the paragraph after Lemma \ref{convergence1} $(iv)$, we can 
  conclude that $u$ is nontrivial. Furthermore, by Lemma \ref{convergence1}, $u$ is a classical solution of $\mathcal{L}u=\lambda(u-1)_+^2f$ in $\Omega\setminus\partial\{u>1\}$ and the free boundary condition $|\nabla_{\mathbb{G}}u^+|^2-|\nabla_{\mathbb{G}}u^-|^2=2$ in the
   sense of \eqref{weak_conv_14}, plus it vanishes on the boundary $\partial\Omega$. 
\end{proof}
\begin{remark}\label{lim_case_reason}
We note that the limiting conditions in Lemma \ref{convergence1} are still an open problem, which is sublinear in its nature.
\end{remark}
	\section{Appendix: Radon Measure Lemma}\label{s6}
	\begin{lemma}\label{positive_Rad_meas}
		$u \in W_{\text{loc}}^{1,p}(\Omega)$ and the Radon measure $\mathcal{M}=\mathcal{L}u$ is nonnegative and supported on $\Omega\cap\{u<1\}$.
	\end{lemma}
	\begin{proof}
		We follow the idea of the proof in {\sc Alt-Caffarelli} \cite{alt_caffa}. Choose $\delta>0$ and a test function $\varphi^p\chi_{\{u<1-\delta\}},$ where $\varphi\in C_0^{\infty}(\Omega)$. Then
		\begin{align}\label{app_2}
			\begin{split}
				0&=\int_{\Omega}\tilde{\nabla} u\nabla(\varphi^2\min\{u-1+\delta,0\})dx\\
				&=\int_{\Omega\cap\{u<1-\delta\}}\tilde{\nabla} u(\varphi^2\min\{u-1+\delta,0\})dx\\
				&=\int_{\Omega\cap\{u<1-\delta\}}|\nabla_{\mathbb{G}} u|^2\varphi^2dx+\int_{\Omega\cap\{u<1-\delta\}}\varphi(u-1+\delta)\tilde{\nabla} u\varphi dx,
			\end{split}
		\end{align}
		and so by the Caccioppoli like estimate, we have
		\begin{align}
			\begin{split}
				\int_{\Omega\cap\{u<1-\delta\}}|\nabla_{\mathbb{G}} u|^2\varphi^2dx&=-2\int_{\Omega\cap\{u<1-\delta\}}\varphi(u-1+\delta)\tilde{\nabla}u\varphi dx\\
				&\leq c\int_{\Omega}u^2|\nabla_{\mathbb{G}}\varphi|^2dx.
			\end{split}
		\end{align}
		Since $\int_{\Omega}|u|^2dx<\infty$, by passing the limit $\delta\rightarrow 0,$ we can conclude that $u\in W_{\text{loc}}^{1,2}(\Omega)$. Furthermore, for a nonnegative $\zeta\in C_0^{\infty}(\Omega)$ we have
		\begin{align}\label{app_3}
			\begin{split}
				-\int_{\Omega}\tilde{\nabla}u\zeta dx=&\left(\int_{\Omega\cap\{0<u<1-2\delta\}}+\int_{\Omega\cap\{1-2\delta<u<1-\epsilon\}}+\int_{\Omega\cap\{1-\delta<u<1\}}\right.
				\left.+\int_{\Omega\cap\{u>1\}}\right)\\
				&\left[\tilde{\nabla}u\left(\zeta\max\left\{\min\left\{2-\frac{1-u}{\delta},1\right\},0\right\}\right)\right]dx\\
				\geq& \int_{\Omega\cap\{1-2\delta<u<1-\delta\}}\left[\tilde{\nabla}u\left(2-\frac{1-u}{\delta}\right)\zeta+\frac{\zeta}{\delta}|\nabla_{\mathbb{G}} u|^2\right]dx\geq 0.
			\end{split}
		\end{align}
		On passing to the limit $\delta\rightarrow 0,$ we obtain $\mathcal{L}(u-1)_{-}\geq 0$ in the distribution sense. Therefore there exists a Radon measure, say $\mathcal{M},$
		such that $\mathcal{M}=\mathcal{L}(u-1)_{-}\geq 0$.
	\end{proof} 
	\subsection*{Acknowledgements}
	The first author thanks CSIR, India (25(0292)/18/EMR-II), NBHM, India (02011/47/2021/
	NBHM(R.P.)/R\&D II/2615) and the Science and Engineering Research board (MATRICS scheme) - India (MTR/2018/000525) for the financial support. The second author was supported by the Slovenian Research Agency grants P1-0292, N1-0278, N1-0114 and N1-0083. We acknowledge O.H. Miyagaki for  comments and suggestions during the early preparation of the manuscript. The authors also acknowledge several very constructive comments by the anonymous reviewer which have helped to improve the manuscript.


\begin{thebibliography}{99}		
		\bibitem{adam1} Adamowicz, T., Warhurst, B., Mean value property and harmonicity on Carnot-Carath\'{e}odory groups, Potential Anal., 52, 497--525, 2020. 
		
		\bibitem{alt_caffa} Alt, H.W., Caffarelli, L.A., Existence and regularity for a minimum problem with free boundary, J. Reine Angew. Math., 325, 105--144, 1981.
		
		\bibitem{2}  Alt, H.W., Caffarelli, L.A., Friedman, A., Variational problems with two phases and their free boundaries, Trans. Amer. Math. Soc., 282(2), 431--461, 1984.
		
		\bibitem{Batchelor1} Batchelor, G.K., 
		On steady state laminar flow with closed streamlines at large Reynolds number, 
		J. Fluid Mech., 1, 177--190, 1956.
		
		\bibitem{Batchelor2} Batchelor, G.K., 
		A proposal concerning laminar wakes behind bluff bodies at large Reynolds number, 
		J. Fluid mech., 1, 388--398, 1956.	
		
		\bibitem{Caffa_jeri_kenig} Caffarelli, L.A., Jerison, D., Kenig, C.E., Some new monotonicity theorems with applications to free boundary problems, Ann. of Math. (2), 155(2), 369--404, 2002. 
		
		\bibitem{12} Caffarelli, L.A., Jerison, D., Kenig, C.E., Global energy minimizers for free boundary problems and full regularity in three dimensions, in: Noncompact Problems at the Intersection of Geometry, Analysis, and Topology, in: Contemp. Math., vol. 350, Amer. Math. Soc., Providence, RI, pp. 83--97, 2004.
		
		\bibitem{2} Capogna, L., Danielli, D., Garofalo, N., An embedding theorem and the Harnack inequality for nonlinear subelliptic equations, Comm. Partial Differential Equations, 18, 1765--1794, 1993.
		
		\bibitem{chen_square} Chen, H., Chen, H.G., Estimates of Dirichlet eigenvalues for a class of sub-elliptic operators, Proc. Lond. Math. Soc. (3), 122(6), 808--847, 2021. 
		
		\bibitem{7} Danielli, D., Regularity at the boundary for solutions of nonlinear subelliptic equations, Indiana Univ. Math. J., 44, 269--286, 1995.
		
		\bibitem{dipi1} Dipierro, S., Karakhanyan, A., Valdinoci, E., New trends in free boundary problems, Adv. Nonlinear Stud., 17(2), 319--332, 2017.
		
		\bibitem{elc1} Elcrat, A.R., Miller, K.G., Variational formulas on Lipschitz domains, Trans. Am. Math. Soc., 347(7), 2669--2678, 1995.
		
		\bibitem{valdinoci1} Ferrari, F., Valdinoci, E., Density estimates for a fluid jet model in the Heisenberg group, J. Math. Anal. Appl., 382(1), 448--468, 2011.
		
		\bibitem{foll}Folland, G.B., Subelliptic estimates and function spaces on nilpotent Lie groups, 
		Ark. Mat., 13, 16--207, 1975.
		
		\bibitem{ghosh1} Ghosh, S., Kumar, V., Ruzhansky, M., Compact embeddings, eigenvalue problems, and subelliptic Brezis-Nirenberg equations involving singularity on stratified Lie groups, arXiv:2205.06007v1, 2022.
		
		\bibitem{Gil_trud} Gilbarg, D., Trudinger, N.S., Elliptic Partial Differential Equations of Second Order, Springer-Verlag, Berlin Heidelberg, 2001.
		
		\bibitem{14} Hajlasz, P., Koskela, P., Sobolev met Poincar\'{e}, Mem. Amer. Math. Soc., 145(688), 2000.	
		
		\bibitem{hofer} Hofer, H., A geometric description of the neighbourhood of a critical point given by the mountain-pass theorem, J. Lod. Math. Soc. (2), 31(3), 566--570, 1985.
		
		\bibitem{J_P_1} Jerison, D., Perera, K., \newblock A multiplicity result for the Prandtl-Batchelor free boundary problem (Preprint arXiv:2003.05921).		
		
		\bibitem{30}Jerison, D., Perera, K., Higher critical points in an elliptic free boundary problem, J. Geom. Anal., 28(2), 1258--1294, 2018.
			
		
		\bibitem{rajala1} Magnani, V., Rajala, T., Radon-Nikodym property and area formula for Banach homogeneous group targets, Int. Math. Res. Notices, 2014(23), 639--6430, 2014.
		
		\bibitem{Nagel1} Nagel, A., Analysis and geometry on Carnot-Carath\'{e}odory spaces,  	\url{https://people.math.wisc.edu/˜nagel/2005Book.pdf}, 2005.
		
		\bibitem{PRR}
		Papageorgiou, N.S.,  R\u adulescu, V.D.,  Repov\v s, D.D.,
		Nonlinear Analysis - Theory and Methods,
		Springer Monographs in Mathematics, Springer, Cham, 2019.
		
		\bibitem{Perera_NoDEA} Perera, K., 
		On a class of elliptic free boundary problems with multiple solutions,
		Nonlinear Differ. Equ. Appl., 28, Art: 36, 2021.	
		
		\bibitem{ruzh-suru-1} Ruzhansky, M., Suragan, D., Layer potentials, Kac's problem, and refined Hardy inequality on
		homogeneous Carnot groups, Adv. Math., 308, 483--528, 2017.
		
		\bibitem{vodo1} Vodop'yanov, S.K., Weighted Sobolev spaces and the boundary behavior of solutions
		of degenerate hypoelliptic equations, Siberian Math. J., 36, 27--300, 1995.
		
		\bibitem{vodo2} Vodop'yanov, S.K., Chernikov, V.M., Sobolev spaces and hypoelliptic equations,
		Trudy Inst. Mat., 29, 7--62, 1995.	
		
		\bibitem{43} Weiss, G.S., Partial regularity for weak solutions of an elliptic free boundary problem, Comm. Partial Differential Equations 23(3--4), 439--455, 1998.
		
		\bibitem{44} Weiss,G.S., Partial regularity for a minimum problem with free boundary, J. Geom. Anal. 9(2), 317--326, 1999.
		\bibitem{xu} Xu, Ch. J., Subelliptic variational problems, Bull. Soc. Math. France, 118,
		147--169, 1990.		
	
	\end{thebibliography}
\end{document}